\documentclass[12pt]{article}
\usepackage{fullpage}
\usepackage{amssymb,amsmath,amsfonts,amsthm,latexsym,enumerate,url,cases}
\numberwithin{equation}{section}
\usepackage{hyperref}
\hypersetup{colorlinks=true,citecolor=blue,linkcolor=blue,urlcolor=blue}

\newtheorem{theorem}{Theorem}[section] %
\newtheorem{lemma}[theorem]{Lemma} %

\begin{document}
\title{On additive representation functions\footnote{The work is supported by the
National Natural Science Foundation of China, Grant No. 11771211
and a project funded by the Priority Academic Program Development
of Jiangsu Higher Education Institutions.}}

\author{  Yong-Gao Chen\footnote{ygchen@njnu.edu.cn(Y.-G. Chen), 1466810719@qq.com(H. Lv)}, Hui Lv\\
\small  School of Mathematical Sciences and Institute of Mathematics, \\
\small  Nanjing Normal University,  Nanjing  210023,  P. R. China
}
\date{}
\maketitle \baselineskip 18pt \maketitle \baselineskip 18pt

{\bf Abstract.}  Let $A$ be an infinite set of natural numbers.
For $n\in \mathbb{N}$, let $r(A, n)$ denote the number of
solutions of the equation $n=a+b$ with $a, b\in A, a\le b$. Let
$|A(x)|$ be the number of integers in $A$ which are less than or
equal to $x$. In this paper, we prove that, if $r(A, n)\not= 1$
for all sufficiently large integers $n$, then $|A(x)|> \frac 12
(\log x/\log\log x)^2$ for all sufficiently large $x$.



\vskip 5mm

\section{Introduction}

 Let $\mathbb{N}$ be the set of all natural numbers and let $A$ be an infinite set of $\mathbb{N}$.
For $n\in \mathbb{N}$, let $r(A, n)$ denote the number of
solutions of the equation $n=a+b$ with $a, b\in A, a\le b$. Let
$A(x)$ be the set of integers in $A$ which are less than or equal
to $x$. In 1998,  Nicolas, Ruzsa and S\' ark\" ozy \cite{Ruzsa}
proved that there exists an infinite set $A$ of $\mathbb{N}$ and a
positive constant $c$ such that $r(A, n)\not= 1$ for all
sufficiently large integers $n$ and $|A(x)|\le c (\log x)^2$ for
all $x\ge 2$. In \cite{Ruzsa}, it was also proved that, if $A$ is
an infinite set of $\mathbb{N}$ such that $r(A, n)\not= 1$ for all
sufficiently large integers $n$, then
$$\limsup |A(x)| \left( \frac{\log\log x}{\log x}\right)^{3/2} \ge
\frac 1{20}.$$ In 2001, S\' andor \cite{Sandor} disproved a
conjecture of Erd\H os and Freud \cite{ErdosFreud} by constructing
an $A$ such that $r(A, n)\le 3$ for all $n$, but $r(A, n)=1$ holds
only for finitely many values of $n$. In 2004, Balasubramanian and
Prakash \cite{Balasubramanian} showed that there exists an
absolute constant $c>0$ with the following property: for any
infinite set $A$ of $\mathbb{N}$ such that $r(A, n)\not= 1$ for
all sufficiently large integers $n$, then
$$|A(x)|\ge  c \left( \frac{\log x}{\log\log x}\right)^{2}$$
for all sufficiently large  $x$. One can obtain $c=\frac 1{2904}$
from the proof of \cite{Balasubramanian}.

In this paper, the following result is proved.

\begin{theorem}\label{thm1} If $A$ is an infinite subset of $\mathbb{N}$ such that $r(A, n)\not= 1$
for all sufficiently large integers $n$, then $$|A(x)|> \frac 12
\left(\frac{\log x}{\log\log x}\right)^2$$ for all sufficiently
large $x$.
\end{theorem}

The key points in this paper are Lemma \ref{lem2} and Lemma
\ref{lem3}. We believe that Lemma \ref{lem3} will be  useful in
the future in Graph Theory.

\section{Proofs}

In the following, we always assume that $A$ is an infinite subset
of $\mathbb{N}$ and $r(A, n)\not= 1$ for all $n\ge n_0$  and
$a_0\in A$ with $a_0\ge n_0$.

Firstly we give some lemmas.

\begin{lemma}(\cite[Lemma 1]{Balasubramanian})\label{lem1} For
every real number $t\ge a_0$, the interval $(t, 2t]$ contains an
element of the set $A$.
\end{lemma}

\begin{lemma}\label{lem2} If $x$ is a large
number with $$|A(x)|\le \left(\frac{\log x}{\log\log x}\right)^2$$
and $$a_0\le b\le \frac{x}{(\log x)^2},$$ then there exists $a\in
A$ with $a>3b$ and $a+b<x$  such that
$$[a-b, a)\cap A=\emptyset , \quad  |(b, a+b] \cap A|\ge
\frac{a+b}{2b}-1.$$
\end{lemma}
\begin{proof} By Lemma \ref{lem1}, $(b, 2b]\cap A\not=\emptyset $.
Since $$(|A(x)|+2) b\le \left(\frac{\log x}{\log\log
x}\right)^2\frac{x}{(\log x)^2} +2\frac{x}{(\log x)^2}<x, $$
$a_0\le b$ and $a_0\in A$, it follows that $$|(b, (|A(x)|+2) b]
\cap A|< |A(x)|.$$ So there exists an integer $1\le k\le |A(x)|$
such that
$$ (ib, (i+1)b]\cap A\not=\emptyset, \quad i=1,2,\dots , k$$
and  $((k+1)b, (k+2)b]\cap A=\emptyset$. By Lemma \ref{lem1},
$((k+1)b, 2(k+1)b]\cap A\not=\emptyset $. Now we take $a$ to be
the least integer in $((k+1)b, 2(k+1)b]\cap A$. Noting that
$((k+1)b, (k+2)b]\cap A=\emptyset$, we have $a>(k+2)b\ge 3b$ and
$(k+1)b<a-b<a$. It follows that $[a-b, a)\cap A=\emptyset$. It is
clear that
$$a+b\le 2(k+1)b +b\le 5kb\le 5|A(x)|b
\le 5 \left(\frac{\log x}{\log\log x}\right)^2 \frac{x}{(\log
x)^2} <x$$ and \begin{eqnarray*}|(b, a+b] \cap A|&=& \sum_{i=1}^k
| (ib, (i+1)b]\cap A|+|((k+1)b, a+b]|\\
&\ge& k+1=\frac{2(k+1)b+b}{2b}-\frac 12
>\frac{a+b}{2b}-1.\end{eqnarray*}
This completes the proof of Lemma \ref{lem2}.
\end{proof}

\begin{proof}[Proof of Theorem \ref{thm1}] We assume that $x$ is a large number. If
$$|A(x)|> \left(\frac{\log x}{\log\log
x}\right)^2,$$ then we are done. In the following, we assume that
\begin{equation}\label{eqnc2}|A(x)|\le \left(\frac{\log x}{\log\log
x}\right)^2.\end{equation} We will prove that
$$|A(x)|> \frac 12
\left(\frac{\log x}{\log\log x}\right)^2.$$

Let $b_1=a_0$. By Lemma \ref{lem2},  there exists $a_1\in A$ with
$a_1>3b_1$ and $a_1+b_1<x$  such that
$$[a_1-b_1, a_1)\cap A=\emptyset , \quad  |(b_1, a_1+b_1] \cap A|\ge
\frac{a_1+b_1}{2b_1}-1.$$ Let $b_2=a_1+b_1$. Continuing this
procedure, we obtain two sequences $b_1<b_2<\cdots <b_m$ and
$a_1<a_2<\cdots < a_{m}$ with  $a_k>3b_k$, $a_k+b_k<x$ $(1\le k\le
m)$ and $b_k=a_{k-1}+b_{k-1}$ $(2\le k\le m)$ such that
$$[a_k-b_k, a_k)\cap A=\emptyset , \quad  |(b_k, a_k+b_k] \cap A|\ge
\frac{a_k+b_k}{2b_k}-1,\quad  k=1,2, \dots , m,$$ where
$$a_m+b_m>\frac{x}{(\log x)^2}, \quad  b_m=a_{m-1}+b_{m-1}\le \frac{x}{(\log x)^2}.$$
For any $1\le i<j\le m$, by $r(A, a_i+a_j)\not= 1$ we may choose
one pair $c_{i,j}, d_{i,j}\in A$ with $d_{i,j}\not= a_j$ and
$c_{i,j}\le d_{i,j}$ such that
$$a_i+a_j=c_{i,j}+d_{i,j}.$$
Let
$$S_k=\{ c_{i,k} \mid  i<k, d_{i,k}<a_k \} \cup \{ d_{i,k} \mid  i<k, d_{i,k}<a_k
\} ,$$ $$ M_k= \{ i \mid  i<k, d_{i,k}<a_k \} ,$$
$$T_k=\{  d_{i,k} \mid  i<k, d_{i,k}>a_k
\} ,$$ and $$  N_k= \{ i \mid  i<k, d_{i,k}>a_k \} .$$ We will
prove that \begin{equation}\label{eq1}S_k\subseteq A\cap (b_k,
a_k), \quad |S_k|\ge |M_k|\end{equation} and
\begin{equation}\label{eq2}T_k\subseteq A\cap (a_k, a_k+b_k], \quad |T_k|= |N_k|.\end{equation}

For $k=1$, we have $S_k=T_k=\emptyset $ and $M_k=N_k=\emptyset $.
So \eqref{eq1} and \eqref{eq2} hold for $k=1$. Now we assume that
$k\ge 2$.

It is clear that
$$d_{i,k}=a_i+a_k-c_{i,k}\le a_i+a_k\le a_{k-1}+a_k\le b_k+a_k. $$
This implies that $T_k\subseteq A\cap (a_k, a_k+b_k]$. If
$d_{u,k}=d_{v,k} \in T_k$ for some pairs $1\le u<v<k$, then, by
$$a_u+a_k=c_{u,k}+d_{u,k}, \quad a_v+a_k=c_{v,k}+d_{v,k},$$
we have
\begin{eqnarray*} a_v&=&c_{v,k}+d_{v,k} -a_k>c_{v,k}\ge c_{v,k}-c_{u,k}=a_v-a_u\\
&\ge & a_v-a_{v-1}\ge a_v-a_{v-1}-b_{v-1}=a_v-b_v.\end{eqnarray*}
This contradicts $[a_v-b_v, a_v)\cap A=\emptyset $. Thus, if
$d_{u,k}, d_{v,k} \in T_k$ with $1\le u<v<k$, then $d_{u,k}\not=
d_{v,k}$. Hence $|T_k|= |N_k|$. Thus we have proved that
\eqref{eq2} holds.

If $i<k$ and $d_{i,k}<a_k$, then by $[a_k-b_k, a_k)\cap
A=\emptyset$ we have $d_{i,k}<a_k-b_k$. Thus
$$c_{i,k}=a_i+a_k-d_{i,k}>a_k-(a_k-b_k)=b_k.$$
It follows that $S_k\subseteq (b_k, a_k)\cap A$.

To prove $|S_k|\ge |M_k|$, it is convenient to use the language
from graph theory.

\medspace

A graph $G$ consists of two parts: $V=V(G)$ of its vertices and
$E=E(G)$ of its edges, where $E(G)$ is a subset of $\{ \{ u, v\}
\mid u, v\in V \} $. Here we allow $G$ contains loops (i.e. $\{ v,
v\} \in E(G)$) and $G$ is a undirected graph. A nontrivial closed
walk is an alternating sequences of vertices and edges $v_1, e_1,
v_2, \dots , v_{n-1}, e_{n-1}, v_n, e_n, v_1$ such that at least
one of edges appears exactly one time and each edge repeats at
most two times. Furthermore, if $n$ is even, then the nontrivial
closed walk is called a nontrivial even closed walk, otherwise, a
nontrivial odd closed walk. A nontrivial closed walk $v_1, e_1,
v_2, \dots , v_{n-1}, e_{n-1}, v_n, e_n, v_1$ is called a closed
trail if $v_1, v_2, \dots , v_n$ are distinct. Furthermore, if $n$
is even, then the  closed trail is called an even closed trail,
otherwise, an odd closed trail. In these definitions, we allow
$n=1$.

\begin{lemma}\label{lem3}  If a graph $G$ has no nontrivial even closed walk, then
$$|E(G)|\le |V(G)|.$$\end{lemma}

\begin{proof} It is enough to prove the lemma when $G$ is connected.
Since $G$ has no nontrivial even closed walk, it follows that $G$
has no even closed trail.

Suppose that $K$ and $L$ were  two distinct  odd closed trails of
$G$.

If $K$ and $L$ have at least one common vertex $v$, then $K$ and
$L$ can be written as
$$K: v, e, u_1, e_1, \dots , u_m, e_m, v$$
and
$$L: v, e', v_1, e_1', \dots , v_n, e_n', v.$$
Thus
$$K\cup L : v, e, u_1, e_1, \dots , u_m, e_m, v, e', v_1, e_1', \dots , v_n, e_n', v$$
is a nontrivial even closed walk of $G$, a contradiction.

If $K$ and $L$ have no common vertex, then there is a walk $W$
which connects $K$ and $L$ since $G$ is connected. Let $W_0$ be
the shortest walk which connects $K$ and $L$.  Now $K$, $L$ and
$W_0$ can be written as
$$K: u,e,u_1, e_1, \dots , u_m,  e_m, u,$$
$$L: v,e',v_1, e_1', \dots , v_n, e_n', v$$
and
$$W_0: u,e'',w_1, e_1'', \dots , w_t, e_t'', v.$$
Thus
$$K\cup W_0 \cup L \cup W_0:u,e,\dots ,e_m, u,e'', \dots e_t'', v,e', \dots , e_n', v,e_t'',\dots,e'',u$$
is a nontrivial even closed walk of $G$, a contradiction.

Now we have proved that $G$ has at most one   odd closed trail
(includes loops). For any subgraph $H$ of $G$, let $\mu
(H)=|E(H)|-|V(H)|$. Let $H_1$ be a connected subgraph of $G$ with
the least $|V(H_1)|$ such that $\mu (H_1)=\mu (G)$. Since $G$ has
at most one  odd closed trail, it follows that $H_1$ has at most
one  odd closed trail. Thus $H_1$ contains only one vertex or
$H_1$ is an  odd closed trail.  So $\mu (H_1)=-1$ or $0$. That is,
$\mu (G)=-1$ or $0$. Therefore, $|E(G)|\le |V(G)|$. This completes
the proof of Lemma \ref{lem3}.
\end{proof}

\medspace

Now we return to the proof of Theorem \ref{thm1}. If
$S_k=\emptyset $, then $M_k=\emptyset $. In this case,
$|S_k|=|M_k|$. Now we assume that $S_k\not= \emptyset $ and define
a graph $G_k$ such that $V(G_k)=S_k$ and
$$E(G_k)=\{ \{ c_{i,k}, d_{i,k}\} \mid i<k, d_{i,k}<a_k \} .$$ Now
we show that $G_k$ has no nontrivial even closed walk.

Suppose that $G_k$ has a nontrivial even closed walk:$$v_1, e_1,
v_2, \dots , v_{2n-1}, e_{2n-1}, v_{2n}, e_{2n}, v_1.$$ Since $\{
v_i, v_{i+1} \} \in E(G_k)$, there exists $\ell_i <k $ such that
$$v_i +v_{i+1} = a_{\ell_i} +a_k,$$
where $v_{2n+1}=v_1$. Thus
$$\sum_{i=1}^{2n} (-1)^i (a_{\ell_i} +a_k)=\sum_{i=1}^{2n} (-1)^i (v_i
+v_{i+1}) =0.$$ It follows that
$$\sum_{i=1}^{2n} (-1)^i a_{\ell_i}=0.$$
We rewrite this as
$$\sum_{i=1}^{k-1} x_i a_{i} =0.$$
Since at least one of edges appears exactly one time and each edge
repeats at most two times in $e_1, e_2, \dots , e_{2n}$, it
follows that $x_i\in \{ -2, -1, 0, 1, 2\}$ $(1\le i\le k-1)$ and
at least one of $x_i$ is nonzero. Let $j$ be the largest index
such that $x_j\not= 0$. Noting that
$$a_{i+1}>3b_{i+1}=3(a_i+b_i)>3a_i,$$ we have
$$ a_j\le |x_ja_j|=\left| -\sum_{i=1}^{j-1} x_i a_{i}\right| \le 2\sum_{i=1}^{j-1}
a_{i}<2 (\frac 13 +\frac 1{3^2} + \cdots ) a_j =a_j,$$ a
contradiction. Hence $G_k$ has no nontrivial even closed walk. By
Lemma \ref{lem3}, we have $$|M_k|=|E(G_k)|\le |V(G_k)|=|S_k|.$$
Thus we have proved that \eqref{eq1} holds. By \eqref{eq1} and
\eqref{eq2}, we have
\begin{eqnarray*}|A\cap (b_k, a_k+b_k]|&\ge & |A\cap (b_k, a_k)|+|A\cap (a_k, a_k+b_k]|+|\{ a_k\} |\\
&\ge & |S_k|+|T_k|+1\\
&\ge & |M_k|+|N_k|+1=k.\end{eqnarray*} Noting that
$b_{k+1}=a_k+b_k$ for $k=1, 2,\dots, m-1$ and $a_m+b_m<x$, we have
\begin{eqnarray*}|A(x)|&\ge&\sum_{k=1}^{m-1} |A\cap (b_k,
b_{k+1}]| +|A\cap (b_m, a_m+b_m]|\\
&=& \sum_{k=1}^m|A\cap (b_k, a_k+b_k]|\\
&\ge & 1+2+\cdots +m\\
&=& \frac 12 m(m+1).\end{eqnarray*} On the other hand,
\begin{eqnarray*}|A(x)|&\ge & \sum_{k=1}^m
|(b_k, a_k+b_k] \cap A|\\
& \ge & \sum_{k=1}^m \left( \frac{a_k+b_k}{2b_k}-1\right) \\
&=& \sum_{k=1}^{m-1} \frac{b_{k+1}}{2b_k} +\frac{a_m+b_m}{2b_m}-m\\
&\ge & \sum_{k=1}^{m-1} \frac{b_{k+1}}{2b_k} +\frac{x}{2b_m (\log x)^2}-m\\
&\ge & m \left( \frac{x}{2b_m (\log x)^2} \prod_{k=1}^{m-1}
\frac{b_{k+1}}{2b_k}\right)^{1/m}-m\\
&=&\frac 12 m \left( \frac{x}{b_1(\log x)^2} \right)^{1/m}-m.
\end{eqnarray*}
If $$m<\frac 14 \frac{\log x}{\log\log x},$$ then
\begin{eqnarray*}|A(x)|&\ge& \frac 12 m \left( \frac{x}{b_1(\log x)^2} \right)^{1/m}-m \\
&\ge & \frac 12 e^{(\log x -2\log\log x -\log b_1)/m} -m \\
&>&\frac 12 e^{3\log\log x} - \frac 14 \frac{\log x}{\log\log
x}\\
&=&\frac 12 (\log x)^3 - \frac 14 \frac{\log x}{\log\log x}\\
&>&(\log x)^2,\end{eqnarray*} a contradiction with \eqref{eqnc2}.
So
$$m\ge \frac 14 \frac{\log x}{\log\log x}. $$
If $$m< \frac{\log x}{\log\log x} ,$$ then
\begin{eqnarray*}|A(x)|&\ge& \frac 12 m \left( \frac{x}{b_1(\log x)^2} \right)^{1/m}-m \\
&\ge & \frac 18 \frac{\log x}{\log\log x}  \exp \left(\frac{(\log x -2\log\log x -\log b_1)\log\log x}{\log x}\right) -m \\
&= & \frac 18 \frac{\log x}{\log\log x}  \exp \left( \log\log x + \frac{(-2\log\log x -\log b_1)\log\log x}{\log x}\right) -m \\
&=& \frac 18 \frac{(\log x)^2}{\log\log x} (1+o(1))\\
&>&\left(\frac{\log x}{\log\log x}\right)^2,\end{eqnarray*} a
contradiction with \eqref{eqnc2}. Hence
$$m\ge \frac{\log x}{\log\log x}. $$
Therefore,
$$|A(x)|\ge \frac 12 m(m+1)> \frac 12 \left(\frac{\log x}{\log\log
x}\right)^2.$$ This completes the proof.\end{proof}

\end{document}